\documentclass[12pt, leqno]{article}

\usepackage{amsmath,amsthm}
\usepackage{amssymb}

\newtheorem{theorem}{Theorem}[section]
\newtheorem{corollary}[theorem]{Corollary}
\newtheorem{lemma}[theorem]{Lemma}

\theoremstyle{definition}

\newtheorem{remark}[theorem]{Remark}

\numberwithin{equation}{section}

\def\Z{{\mathbb Z}}
\def\R{{\mathbb R}}

\def\I{{\mathcal I}}
\def\L{{\mathcal L}}

\def\S{{\mathcal S}}
\def\T{{\mathcal T}}
\def\W{{\mathcal W}}

\begin{document}

\baselineskip=17pt

\title{The larger sieve and polynomial congruences}

\author{Patrick Letendre\\
D\'ep. de  math\'ematiques et de statistique\\ 
Universit\'e Laval\\
Qu\'ebec\\
Qu\'ebec G1V 0A6\\
Canada\\
Patrick.Letendre.1@ulaval.ca}

\date{}

\maketitle

\renewcommand{\thefootnote}{}



\renewcommand{\thefootnote}{\arabic{footnote}}
\setcounter{footnote}{0}

\begin{abstract}
We obtain a small improvement of Gallagher's larger sieve and we extend it to higher dimensions. We also obtain two interesting upper bounds for the number of solutions to polynomial congruences.
\end{abstract}

\section{Introduction}

In his paper of 1971, Gallagher introduced a new tool in number theory that is now known as {\it the larger sieve} and also as {\it Gallagher's larger sieve}. As indicated by its name, it is a complementary inequality to the large sieve. More precisely, let $\S$ be a set of integers in an interval of length $M$ for which there exists a set $\mathcal{Q}$ of prime powers $q=p^{\alpha_p}$ such that each numbers $n \in \S$ belong to at most $\nu(q)$ congruence classes modulo $q$. Then
\begin{eqnarray}\label{gall}
\# \S \le \frac{\sum_{q \in \mathcal{Q}} \Lambda(q)-\log M}{\sum_{q \in \mathcal{Q}}\frac{\Lambda(q)}{\nu(q)} - \log M}
\end{eqnarray}
holds if the denominator is positive. Here, $\Lambda(\cdot)$ denote the classical von Mangoldt function, that is
$$
\Lambda(q) = \left\{\begin{array}{cl} \log p & \mbox{if}\ q=p^j\ \mbox{for some prime}\ p\ \mbox{and}\ j \ge 1,\\
0 & \mbox{otherwise}.
\end{array}\right.
$$
Inequality \eqref{gall} has been proven to be stronger than the large sieve when most of the values of $\nu(q)$ are small, see \cite{pxg}. In the book \cite{jf:hi}, the authors propose a generalization of the larger sieve. Some related results and a discussion can be found in \cite{esc:ce}.

A seemingly unrelated subject with above is the study of polynomial congruences. Let $f(x) := a_nx^n + \cdots + a_1x + a_0 \in \Z[x]$ be a polynomial of degree $n \ge 2$ and $q \ge 2$ be an integer satisfying
$$
\gcd(a_n,\dots,a_0,q)=1.
$$
It has been established in \cite{svk} that the number of solutions $N(f,q)$ to the equation
\begin{eqnarray}
f(x) \equiv 0 \pmod{q} \qquad (x=1,\dots,q)
\end{eqnarray}
satisfies
\begin{eqnarray}
N(f,q) \le \left(\frac{n}{e}+O((\log n)^2)\right)q^{1-\frac{1}{n}}.
\end{eqnarray}
It is also shown to be essentially best possible since there are infinitely many polynomials $f(x)$ and values of $q$ for which
$$
N(f,q) > \left(\frac{n}{e}+c_1\log n\right)q^{1-\frac{1}{n}}
$$
for some constant $c_1 > 0$.

It raises the question: How many solutions $x$ of $f(x) \equiv 0 \pmod{q}$ can we find in an interval $\I$ of length $q^{\frac{1}{n}}$? In Theorem 3 of \cite{svk:ts}, an answer has been given and it is of the shape $\ll \log q$. By studying the argument of the demonstration of this theorem, we have been led to a small improvement in the case where $n$ is considered as fixed. Also, our research has led us to an improvement of the inequality \eqref{gall} as well as a generalization to higher dimensions.

Throughout the paper, we often write $\S$, with or without subscript, to denote a set of integer points in $\mathbb{Z}^m$ for some $m \ge 1$. When it is the case, we often write $S$ to denote $\#\S$ with the same subscript. For any integer $q \ge 1$, the functions $\phi(q)$ and $\omega(q)$ are respectively the Euler's phi function and the number of distinct prime divisors of $q$. For any integer $q \ge 1$ and prime $p$, let's denote by $\mathrm{v}_p(q)$ the unique integer $\alpha_p \ge 0$ for which $p^{\alpha_p} \| q$. For two integers $q$ and $\Delta$ and a real number $\alpha$, we write $q^\alpha \mid \Delta$ to signify that $\alpha\mathrm{v}_p(q) \le\mathrm{v}_p(\Delta)$ for each primes $p$.

\section{Statements of theorems}

For each integer $s \ge 2$, let's define
$$
c_s:=\prod_{j=1}^{s}\left(\frac{(j-1)^{2(j-1)}j^j}{(s+j-2)^{s+j-2}}\right)^{\frac{1}{s(s-1)}}.
$$
\begin{theorem}\label{thm:1}
Let $\S$ be a set of integers in the interval $[N,N+M]$ with $M > 0$. Let also $\mathcal{Q}$ be a finite set of pairwise coprime integers. Suppose that for each $q \in \mathcal{Q}$ the integers $n \in \S$ belong to at most $1 \le \nu(q) \le q$ congruence classes modulo $q$. Then, if the denominator is positive, the inequality
\begin{eqnarray}\label{thm_1}
S \le \frac{\sum_{q \in \mathcal{Q}} \log q-\log(c_SM)}{\sum_{q \in \mathcal{Q}}\frac{\log q}{\nu(q)} - \log(c_SM)}
\end{eqnarray}
holds.
\end{theorem}

\begin{remark}
We show in Lemma \ref{lem2} that $c_s$ is essentially $\frac{1}{4}+\epsilon(s)$ so that the first term in the denominator of \eqref{thm_1} can be about $\log 4 = 1.386\dots$ smaller than in \eqref{gall} and still have the inequality effective. One can see directly from the proof that inequality \eqref{thm_1} is at least as good as \eqref{gall} provided $S \ge \max_{q \in \mathcal{Q}}\nu(q)$. Also, an inequality like \eqref{thm_1} can be stated with the function $\Lambda(\cdot)$ replacing $\log(\cdot)$ in both sums, in which case we get an inequality that is always at least as good as \eqref{gall}.
\end{remark}

\begin{corollary}
Assume that we are in the situation of Theorem \ref{thm:1}. We either have $S \le 1243$ or we have that
\begin{equation}\label{cor_1-const}
S < \frac{\sum_{q \in \mathcal{Q}} \Lambda(q)-\log M+1.38}{\sum_{q \in \mathcal{Q}}\frac{\Lambda(q)}{\nu(q)} - \log M+1.38}
\end{equation}
holds if the denominator is positive.
\end{corollary}

\begin{remark}
It is possible to show that an inequality like \eqref{cor_1-const} cannot hold if the constant is too large. In fact, using the polynomial $P(x)=x^2+x$, one can show that the optimal constant has to be less than
$$
2-\log(2)+2\gamma+4\sum_{p \ge 3}\frac{\log(p)}{p^2-1} \le 3.817.
$$
\end{remark}

Let $v_1,\dots,v_{m+1}$ be points in $\mathbb{R}^m$. We define the quantity
$$
D(v_1,\dots,v_{m+1}) := \begin{Vmatrix} 1 & 1 & \cdots & 1 \\ v_1 & v_2 & \cdots & v_{m+1} \end{Vmatrix}.
$$
The points $v_1,\dots,v_{m+1}$ are in the same (affine) hyperplane if and only if
$$
D(v_1,\dots,v_{m+1})=0.
$$

Let $\Gamma \subseteq \mathbb{Z}^m$ be a lattice in $\mathbb{R}^m$. We denote by $|\Gamma|$ the $m$-dimensional volume of the fundamental parallelepiped of the lattice $\Gamma$. For a fixed set $\Omega \in \R^m$, we write
$$
t(\Omega) := \sup_{v_1,\dots,v_{m+1} \in \Omega}D(v_1,\dots,v_{m+1}).
$$

\begin{theorem}\label{thm:2}
Let $\S$ a set of integer points included in a set $\Omega \in \R^m$ $(m \ge 2)$ of nonzero $m$-dimensional volume. Let also $\L$ be a set of lattices $\Gamma \subset \mathbb{Z}^m$. Suppose that for each lattices $\Gamma \in \L$, the points of $\S$ belong to at most $\nu(\Gamma)$ equivalence classes of $\Z^m/\Gamma$ and that
\begin{equation}\label{thm3cond}
\min_{\substack{v_1,\dots,v_{m+1} \in \S \\ v_i \neq v_j\ for\ i \neq j}} D(v_1,\dots,v_{m+1}) > 0.
\end{equation}
Suppose also that the values of $|\Gamma|$ are pairwise coprime. Then,
\begin{equation}\label{thm_2}
S < \max\left(\gamma_m\max_{\substack{1 \le s \le m \\ 2 \nmid s}} \left(\frac{\sum_{\Gamma \in \L} \frac{\log |\Gamma|}{v(\Gamma)^{m-s}}-\log t(\Omega)}{\sum_{\Gamma \in \L} \frac{\log |\Gamma|}{v(\Gamma)^{m}}-\log t(\Omega)}\right)^{1/s},(m+1)\max_{\Gamma \in \L}v(\Gamma)\right)
\end{equation}
if
$$
\sum_{\Gamma \in \L} \frac{\log |\Gamma|}{v(\Gamma)^{m}}-\log t(\Omega) > 0.
$$
We have set $\gamma_m:=\left\lfloor\frac{m+1}{2}\right\rfloor\frac{m(m+1)}{2}$.
\end{theorem}

\begin{remark}
The hypothesis \eqref{thm3cond} is really strong and seems difficult to deal with in practice. For this reason, we have included Lemma \ref{lem7}. We have also included in Lemma \ref{lem5} an estimate for the value of $t(\Omega)$ in the case where $\Omega$ is a $m$-dimensional parallelepiped.
\end{remark}

Finally, our considerations of the initial problem have led us to the following theorem. It is an improvement of Theorem 3 of \cite{svk:ts}.

\begin{theorem}\label{thm:3}
Consider the polynomial $P(x) := a_nx^n + \cdots + a_1x + a_0 \in \Z[x]$ of degree $n$ and $q \ge 2$ be an integer satisfying $\gcd(a_n,\dots,a_0,q)=1$. Let $\I$ be an interval of length at most $q^{1/n}$. The number $W$ of solutions to the congruence
\begin{eqnarray}\label{thm_3-sys}
P(x) \equiv 0 \pmod{q} \qquad (x \in \I)
\end{eqnarray}
satisfies
\begin{eqnarray}\label{thm3-res}
W \le 2(n-1)^2\omega(q).
\end{eqnarray}
\end{theorem}

\begin{corollary}
Consider the polynomial $P(x) := a_nx^n + \cdots + a_1x + a_0 \in \Z[x]$ of degree $n$ and $q \ge 2$ be an integer satisfying $\gcd(a_n,\dots,a_0,q)=1$. Let $\I$ be an interval of length $L$. The number $W$ of solutions to the congruence
\begin{eqnarray}
P(x) \equiv 0 \pmod{q} \qquad (x \in \I)
\end{eqnarray}
satisfies
\begin{eqnarray}
W \le 2(n-1)^2\omega(q)\left(\frac{L}{q^{1/n}}+1\right).
\end{eqnarray}
\end{corollary}

We also have a modest improvement of Theorem \ref{thm:3} in a very particular case.

\begin{theorem}\label{thm:4}
Consider the polynomial
\begin{eqnarray}\label{form}
P(x):=x^n + d
\end{eqnarray}
of degree $n \ge 2$ with $d \in \Z$. Let $q \ge 2$ be an integer and $\I$ be an interval of length at most $q^{1/n}$. The number $W$ of solutions to the congruence
\begin{eqnarray}\label{thmcong}
P(x) \equiv 0 \pmod{q} \qquad (x \in \I)
\end{eqnarray}
satisfies
\begin{eqnarray}\label{thm_4}
W \le n\omega(q).
\end{eqnarray}
\end{theorem}

\section{Preliminary lemmas}

\begin{lemma}\label{lem1}
For each $s \ge 2$, we have
$$
\max_{0 \le \xi_1 \le \dots \le \xi_s \le 1}\prod_{1 \le i < j \le s}(\xi_j-\xi_i)=c_s^{\binom{s}{2}}.
$$
\end{lemma}

\begin{proof}
This is a restatement of Theorem $8.5.2$ of \cite{gea:ra:rr} with $p=q=0$.
\end{proof}

\begin{lemma}\label{lem2}
For each $s \ge 2$, the inequality
$$
c_s < \frac{1}{4}\exp\left(\frac{s\log(2s)+\frac{1}{4}\log(s)}{s(s-1)}\right)
$$
holds.
\end{lemma}

\begin{proof}
We proceed by induction. We start by checking that the result is true for $2 \le s \le 199$. Now, we write
$$
a_s:=s(s-1)\log c_s
$$
and
$$
f(s):=-s(s-1)\log 4 +s\log 2s+\frac{\log s}{4}-\frac{1}{s}.
$$
We verify that $a_{200} \le f(200)$. For $s \ge 200$, we suppose that $a_{s} \le f(s)$ and we want to establish that $a_{s+1} \le f(s+1)$. It is enough to establish that
\begin{equation}\label{cs0}
a_{s+1}-a_{s} \le f(s+1)-f(s).
\end{equation}
We have
\begin{equation}\label{cs1}
{\scriptstyle a_{s+1}-a_{s}=(s+1)\log (s+1)+(s-1)\log(s-1)-(2s-1)\log(2s-1)-2s\log2}
\end{equation}
and
\begin{equation}\label{cs2}
{\scriptstyle f(s+1)-f(s) = -2s\log 4+\log 2+(s+1)\log(s+1)-s\log s +\frac{1}{4}\log\left(1+\frac{1}{s}\right)+\frac{1}{s(s+1)}.}
\end{equation}
Comparing \eqref{cs1} with \eqref{cs2}, we observe that \eqref{cs0} holds if and only if
\begin{equation}\label{cs4}
g(2s-1)-g(s-1) \le \frac{1}{4}\log\left(1+\frac{1}{s}\right)+\frac{1}{s(s+1)}
\end{equation}
holds, where we have written $g(x):=x\log\left(1+\frac{1}{x}\right)$. Now, we make use of the inequality
$$
\frac{1}{x}-\frac{1}{2x^2} \le \log\left(1+\frac{1}{x}\right) \le \frac{1}{x}-\frac{1}{2x^2}+\frac{1}{3x^3}\quad(x>1),
$$
to establish that the inequality \eqref{cs4} holds if
$$
0 \le \frac{1}{4s}+\frac{1}{2(2s-1)}-\frac{1}{2(s-1)}+\frac{1}{s(s+1)}-\frac{1}{8s^2}-\frac{1}{3(2s-1)^2}
$$
holds. We verify that this is the case for $s \ge 200$, which completes the induction step.
\end{proof}

For a fixed $n \ge 2$, we consider the multiplicative function $g(n,q)$, i.e.
$$
g(n,q)=\prod_{p^j\|q}g(n,p^j),
$$
defined by
$$
g(n,p^j):=\left\{\begin{array}{ll} \gcd(n,\phi(p^j))&\mbox{if}\ p\ge3,\\ 1&\mbox{if}\ p^j=2,\\ \gcd(n,2)&\mbox{if}\ p^j=4,\\ \gcd(n,2)\cdot \gcd(n,\phi(2^{j-1}))&\mbox{if}\ p=2\ \mbox{and}\ j\ge3. \end{array}\right.
$$
In particular, $g(n,q)\le2n^{\omega(q)}$.

\begin{lemma}\label{lem3}
Let $P(x)$ be the polynomial \eqref{form} with $n \ge 2$. Let also $q \ge 2$ be an integer satisfying $\gcd(d,q)=1$. Then the total number of solutions $\pmod{q}$ to the congruence
$$
P(x) \equiv 0 \pmod{q}
$$
is of at most $g(n,q)$.
\end{lemma}

\begin{proof}
The proof is an easy exercise that uses a primitive root of $\left(\mathbb{Z}/p^\alpha\mathbb{Z}\right)^*$, for any odd prime $p$ and $\alpha \ge 1$, together with the fact that any element of $\left(\mathbb{Z}/2^\alpha\mathbb{Z}\right)^*$, with $\alpha \ge 2$, has a unique representation as $(-1)^a5^b$ with $a \in \{0,1\}$ and $b \in \{1,\dots,2^{\alpha-2}\}$. The multiplicativity follows from the Chinese remainder theorem.
\end{proof}

\begin{lemma}\label{lem4}
Let $P(x) := a_nx^n + \cdots + a_1x + a_0 \in \Z[x]$ be a polynomial of degree $n \ge 1$ and $q \ge 2$ be an integer satisfying $\gcd(a_n,\dots,a_0,q)=1$. Let also $x_1 < x_2 < \cdots < x_s$ be a sequence of solutions to the congruence
$$
P(x) \equiv 0 \pmod{q}.
$$
Consider the product
$$
\Delta := \prod_{1 \le i < j \le s}(x_j-x_i).
$$
If $s \ge n+1$ then
$$
q^{\frac{s^2}{2n}-\frac{s}{2}} \mid \Delta.
$$
\end{lemma}

\begin{proof}
This result is proved in Lemma 2.5 of \cite{pl} for $n \ge 2$ and it is clear for $n=1$.
\end{proof}

\begin{lemma}\label{lem5}
Let $\Omega \in \R^m$ be a closed parallelepiped of nonzero $m$-dimensional volume. Then
\begin{eqnarray}\label{lem4_ineq}
t(\Omega) \le \frac{(m+2)^{\frac{m+1}{2}}}{2^m}Vol(\Omega).
\end{eqnarray}
Also, we have that $t(1)=t(2)=Vol(\Omega)$, $t(3)=2Vol(\Omega)$ and that $t(4)=3Vol(\Omega)$.
\end{lemma}

\begin{proof}
It is enough to prove the result for the cube $[0,1]^m$. This is a situation that is similar to a famous problem, see \cite{jb:lc}. Let $A=\begin{pmatrix} 1 & 1 & \cdots & 1 \\ v_1 & v_2 & \cdots & v_{m+1} \end{pmatrix}$ be a matrix that realizes an extremum of the function $\det A$. Suppose at first that one of the vectors $v_j=(a_{2,j},\dots,a_{m+1,j})^t$ has a coordinate $0 < a_{i,j} <1$. We then deduce that
$$
0 = \frac{d}{dx_{i,j}}\det A\left|_{x_{i,j}=a_{i,j}}\right.= (-1)^{i+j}\det A_{i,j}
$$
where $x_{i,j}$ is a variable in position $(i,j)$ in $A$, where the last equality follows by expanding using the $j$-th column and where $A_{i,j}$ is the submatrix $m \times m$ obtained by removing the $i$-th row and the $j$-th column. We deduce that $\det A_{i,j}=0$ so that $\det A$ remains invariant by a modification of the entry $a_{i,j}$. We therefore consider the new matrix $A_1$ for which $a_{i,j}=0$ and all the other entries are the same as in the matrix $A$. We repeat this process until we get to a matrix $A^\prime$ composed only of 0 and 1.

Now, to obtain inequality \eqref{lem4_ineq}, we consider the matrix
$$
B:=\begin{pmatrix} 1 & 0 & \cdots & 0 \\ 1 & 1 & \cdots & 1 \\ x & v_1 & \cdots & v_{m+1}  \end{pmatrix}
$$
where $x=(\frac{1}{2},\dots,\frac{1}{2})^t$. We observe that $\det A=\det B$ and the result follows by subtracting the first column from the others and by using Hadamard's inequality on the rows. The other statements can be verified directly with a computer. The proof is completed.
\end{proof}

\begin{lemma}\label{lem6}
Let $P(x) := x (x-1) \cdots (x-d+1)$ be a polynomial of degree $d \ge 3$. Let also $x_1,\dots,x_n,X$ be positive real numbers satisfying $x_1+\cdots+x_n=X$ and $X \ge dn$. Then,
\begin{eqnarray}\label{lem5.exp}
P(x_1)+\dots+P(x_n) \ge nP\left(\frac{X}{n}\right).
\end{eqnarray}
\end{lemma}

\begin{proof}
Clearly $0 \le x_i \le X$ for each $i=1,\dots,n$ and we can assume that $x_1 \ge x_2 \ge \dots \ge x_n$. Let $j+1$ be the number of nonzero values of $x_i$. Suppose that $j \ge 1$ and consider the function
$$
F(z_1,z_2,\dots,z_{j}):=P(z_1)+\dots+P(z_{j})+P(X-z_1-\cdots-z_{j}).
$$
If $F$ reaches a local extremum at $(x_1,\dots,x_j)$, then
$$
{\scriptstyle 0=\frac{d}{dz_{i}}F(z_1,\dots,z_{j})\Big\vert_{z_1=x_1,\dots,z_{j}=x_{j}}=P'(x_i)-P'(X-x_1-\cdots-x_j) \quad (i=1,\dots,j).}
$$
We deduce that
\begin{eqnarray}\label{lem5.1}
P'(x_1)=\dots=P'(x_{j+1}).
\end{eqnarray}
One can establish the inequality
$$
\max_{x \in [0,d-1]} \prod_{\substack{i=0 \\ i \neq k}}^{d-1}|x-i|  \le (d-1)! \qquad (k=0,\dots,d-1).
$$
We deduce that $\max _{x \in [0,d-1]}|P'(x)| \le d!$ and consequently $|P'(x)| < P'(d)$ for each $x \in [0,d)$. For $x \ge d$, the function $P'(x)$ is strictly increasing. Now, since $x_1+\dots+x_{j+1}=X$, we must have $\max_{i}x_i \ge \frac{X}{j+1}$. It follows that if $\frac{X}{j+1}\ge d$, then \eqref{lem5.1} implies that $x_1=\dots=x_{j+1}=\frac{X}{j+1}$. We have therefore shown that the minimum of the left expression in \eqref{lem5.exp} is of the form $(j+1)P\left(\frac{X}{j+1}\right)$ for a value of $j=0,\dots,n-1$. We then notice that
$$
\frac{d}{dt}tP\left(\frac{X}{t}\right)=P\left(\frac{X}{t}\right)-\frac{X}{t}P'\left(\frac{X}{t}\right)<0
$$
if $\frac{X}{t}>d-1$. The proof is thus completed.
\end{proof}

\begin{lemma}\label{lem7}
Let $\mathcal{N}$ be a finite set of points in $\mathbb{R}^m$. Let $\S \subseteq \mathcal{N}$ a subset of maximal cardinality for which
$$
\min_{\substack{v_1,\dots,v_{m+1} \in \S \\ v_i \neq v_j\ for\  i \neq j}} D(v_1,\dots,v_{m+1}) > 0.
$$
Let also $K$ be the maximal number of points in $\mathcal{N}$ that are all included in an hyperplane. Then
$$
\#\mathcal{N} \le K\max\left(1,\binom{S}{m}\right).
$$
\end{lemma}

\begin{proof}
If such a set $\S$ does not exist, then we have $\#\mathcal{N} \le K$. Otherwise, since $\S$ is a set of maximal cardinality, it follows that each point $v \in \mathcal{N} \setminus \S$ is included in an hyperplane defined by at least one set of $m$ points of $\S$. There are $\binom{S}{m}$ such sets of $m$ points. By hypothesis, each of these sets defines a distinct hyperplane and then each such hyperplane contains at most $K$ points of $\mathcal{N}$. The result follows.
\end{proof}

\section{Proof of Theorem \ref{thm:1}}

Let's denote by $x_1,\dots,x_{S}$ the ordered list of numbers in $\S$. We then consider the product
$$
\Delta:= \prod_{1 \le i < j \le S}(x_j-x_i).
$$
On the one hand, using Lemma \ref{lem1}, we have
\begin{eqnarray*}
\Delta & = & M^{\binom{S}{2}}\prod_{1 \le i < j \le S}\left(\frac{x_j-x_i}{M}\right)\\
& \le & M^{\binom{S}{2}}\max_{0 \le \xi_1 \le \dots \le \xi_S \le 1}\prod_{1 \le i < j \le S}(\xi_j-\xi_i)\\
& = & (c_S M)^{\binom{S}{2}}.
\end{eqnarray*}
On the other hand, let's fix an integer $q \in \mathcal{Q}$ and partition the set $\S$ into the $\nu(q)$ disjoint subsets $\S_r$ that contain the ordered set of numbers $x_{r,1},\dots,x_{r,S_r}$ from $\S$ that belong to the same congruence class modulo $q$. We then write
$$
\Delta_r:= \prod_{1 \le i < j \le S_r}(x_{r,j}-x_{r,i})
$$
and notice that
$$
q^{\binom{S_1}{2}+\cdots+\binom{S_{\nu(q)}}{2}} \mid \Delta_1 \cdots \Delta_{\nu(q)} \mid \Delta.
$$
Now, we find
\begin{eqnarray*}
\sum_{r=1}^{\nu(q)}\binom{S_r}{2} & = & \frac{1}{2}\sum_{r=1}^{\nu(q)}S^2_r-\frac{S}{2}\\
& \ge & \frac{S^2}{2\nu(q)}-\frac{S}{2}
\end{eqnarray*}
using Cauchy-Schwarz's inequality. Since the values of $q \in \mathcal{Q}$ are pairwise coprime, we get to
$$
\prod_{q \in \mathcal{Q}}q^{\frac{S^2}{2\nu(q)}-\frac{S}{2}} \mid \Delta \le (c_S M)^{\binom{S}{2}}.
$$
The result easily follows.

\section{Proof of Theorem \ref{thm:2}}

For a fixed $m \ge 2$, we write the sequence of integer points in $\S$ as $v_1, \dots, v_{S}$ and consider the product
$$
\Delta := \prod_{1 \le i_1 < \dots < i_{m+1} \le S}D(v_{i_1},\dots,v_{i_{m+1}}).
$$
Clearly,
$$
\Delta \le t(\Omega)^{\binom{S}{m+1}}.
$$

Now, let's fix a lattice $\Gamma \in \L$ and partition the set $\S$ into the $\nu(\Gamma)$ disjoint subsets $\S_r$ that contain the set of integer points $v_{r,1},\dots,v_{r,S_r}$ from $\S$ that belong to the same equivalence class of $\Z^m / \Gamma$. We then define
$$
\Delta_r :=\prod_{1 \le i_1 < \dots < i_{m+1} \le S_r}D(v_{r,i_1},\dots,v_{r,i_{m+1}}).
$$
and notice that
$$
|\Gamma|^{\binom{S_1}{m+1}+\cdots+\binom{S_{v(\Gamma)}}{m+1}} \mid \Delta_{1}\cdots\Delta_{v(\Gamma)} \mid \Delta.
$$
From Lemma \ref{lem6} and the hypothesis $S \ge (m+1)v(\Gamma)$ (otherwise \eqref{thm_2} is trivial), we get
$$
|\Gamma|^{v(\Gamma)\binom{S/v(\Gamma)}{m+1}} \mid \Delta.
$$
By assumption the values of $|\Gamma|$ are pairwise coprime and the inequality $S \ge (m+1)v(\Gamma)$ holds for each $\Gamma$. We deduce that
$$
\prod_{\Gamma}|\Gamma|^{v(\Gamma)\binom{S/v(\Gamma)}{m+1}} \le t(\Omega)^{\binom{S}{m+1}}.
$$
We take the logarithm and send everything to the left hand side. We get the inequality
$$
a_{m}b_{m}S^{m}-a_{m-1}b_{m-1}S^{m-1}+\cdots+(-1)^{m}a_0b_0  \le 0
$$
where
$$
x(x-1)\cdots(x-m)=a_{m}x^{m+1}-a_{m-1}x^m+\cdots(-1)^{m}a_0x
$$
and where
$$
b_i:=\sum_{\Gamma \in \L}\frac{\log |\Gamma|}{v(\Gamma)^i}-\log t(\Omega).
$$
The hypothesis $b_m > 0$ implies that $b_i > 0$ for each $i=0,\dots,m$. We deduce that
$$
S \le \left\lfloor\frac{m+1}{2}\right\rfloor \max_{\substack{1 \le s \le m \\ 2 \nmid s}}\frac{a_{m-s}b_{m-s}}{b_{m}S^{s-1}}
$$
from which the result follows after a simple computation.

\section{Proof of Theorem \ref{thm:3}}

From the proof of Lemma 2.5 of \cite{pl}, we know that we can assume that $P(x)=\prod_{j=1}^{n}(x-a_j)$. Also, we can assume that $a_1=0$. 

Step 1: We have an integer $q \ge 2$, a polynomial $P(x)$ of degree $n$ and an interval $\I$ of length $\le q^{1/n}$. We want to find an upper bound for the number of solutions $W$ to the system \eqref{thm_3-sys}. Let's fix a prime power $q_1 = p^{\alpha} \| q$ for which $q_1 \ge q^{1/\omega(q)}$. We consider two cases.

Case 1: The solutions $\W$ to \eqref{thm_3-sys} are in exactly $2 \le t \le n$ congruence classes modulo $p$. Consider a congruence class, say $\ell \pmod{p}$, that has the most solutions, a set we denote by $\W'$. We have $\#\W \le ts$ where $\#\W'=s$. We can assume that $s \ge 2$ since otherwise \eqref{thm3-res} holds. Let's define the polynomial
$$
P_\ell(x):=\prod_{\substack{1 \le j \le n \\ a_j \equiv \ell \pmod{p}}}(x-a_j).
$$
We remark that $P_\ell(x)$ is of degree at most $n+1-t$. Now, we write the solutions in $\W'$ as $x_1 < \cdots < x_s$ and define
$$
\Delta := \prod_{1 \le i < j \le s}(x_j-x_i).
$$
Clearly, $\Delta \le q^{\frac{s^2-s}{2n}}$. Also, using Lemma \ref{lem4} for the polynomials $P(x)$ and $P_\ell(x)$, we get
$$
q_1^{\frac{s^2}{2(n+1-t)}-\frac{s}{2}}\left(\frac{q}{q_1}\right)^{\frac{s^2}{2n}-\frac{s}{2}}=q_1^{\frac{s^2}{2}\frac{t-1}{n(n+1-t)}}q^{\frac{s^2}{2n}-\frac{s}{2}} \mid \Delta.
$$
We deduce that
$$
q^{\frac{s^2}{2\omega(q)}\frac{t-1}{n(n+1-t)}}q^{\frac{s^2}{2n}-\frac{s}{2}} \le q_1^{\frac{s^2}{2}\frac{t-1}{n(n+1-t)}}q^{\frac{s^2}{2n}-\frac{s}{2}} \le q^{\frac{s^2-s}{2n}}
$$
so that
$$
W \le ts \le \frac{t}{t-1}\left(1-\frac{1}{n}\right)n(n+1-t)\omega(q)
$$
and the result follows.

Case 2: The solutions $\W$ to \eqref{thm_3-sys} are in only one congruence class modulo $p$. In this case, since $P(0) \equiv 0 \pmod{q}$, we have that this class is $0 \pmod{p}$. Also, we must have $p \mid a_i$ for $i=1,\dots,n$. Writing $x=pz$, we get
$$
P(x) \equiv 0 \pmod{q} \quad \Longrightarrow \quad P_1(z) \equiv 0 \pmod{\frac{q}{p^{\min(\alpha,n)}}}
$$
where $P_1(z)=\prod_{j=1}^{n}(z-a_{j,1})$ and $a_{j,1}=\frac{a_j}{p}$. We have thus transformed our problem into another one with the integer $q'=\frac{q}{p^{\min(\alpha,n)}}$, the polynomial $P_1(x)$ and an interval of length $\frac{q^{1/n}}{p} \le q'^{1/n}$.

Step 2: If $q' \ge 2$ we return to Step 1 with $q'$ instead of $q$, $P_1(x)$ instead of $P(x)$ and $\I_1$ of length $\le q'^{1/n}$ instead of $\I$. If we are not in Case 1 at some stage, we will get to $q'=1$ and $\I_1$ of length at most $1$ so that $W \le 2$. The proof is completed.

\begin{remark}
We can also proceed as in the proof of Theorem \ref{thm:1} to find an upper bound for $W$. We write the solutions of \eqref{thm3-res} as $x_1 < \cdots < x_W$ and define
$$
\Delta := \prod_{1 \le i<j \le W}(x_j-x_i).
$$
Proceeding as usual and using Lemma \ref{lem4}, we get to
$$
q^{\frac{W^2}{2n}-\frac{W}{2}} \mid \Delta \le (c_W q^{\frac{1}{n}})^{\frac{W^2-W}{2}}.
$$
Thus we have
$$
\left(\frac{1}{c_W}\right)^{W-1} \le q^{1-\frac{1}{n}}
$$
and we deduce from Lemma \ref{lem2} that
$$
W-\frac{\log W}{\log 4}-\frac{\log W}{4W\log 4} \le \left(1-\frac{1}{n}\right)\frac{\log q}{\log 4}+\frac{3}{2}.
$$
Now, we write $F(x):=x-\frac{\log x}{\log 4}-\frac{\log x}{4x\log 4}$ and show that $F'(x)>0$ for $x \ge 1$ and that
$$
F\left(x+\frac{\log x}{\log 4}+\frac{2}{3}\right) \ge x \qquad (x \ge \frac{7}{4}).
$$
From there we get
$$
W < \left(1-\frac{1}{n}\right)\frac{\log q}{\log 4}+\frac{\log\log q}{\log 4}+3.
$$
\end{remark}

\section{Proof of Theorem \ref{thm:4}}

We can assume that $d \in \{1,\dots,q\}$. We first show that it is enough to prove the theorem with the supplementary assumption $\gcd(d,q)=1$. Indeed, assume that $\gcd(d,q)=r$. Let's define the function
$$
\gamma_n(r):=\prod_{p^{\alpha} \| r}p^{\lceil\frac{\alpha}{n}\rceil}.
$$
Each solutions $x \in \I$ of \eqref{thmcong} must also satisfy $\gamma_n(r) \mid x$, Thus, by writing $x=\gamma_n(r)z$, we get to the congruence
$$
\gamma_n(r)^nz^n+d \equiv 0 \pmod{q} \quad \Longrightarrow \quad \frac{\gamma_n(r)^n}{r}z^n+\frac{d}{r} \equiv 0 \pmod{\frac{q}{r}}.
$$

Case 1: $\frac{q}{r} > 1$. If $\gcd\bigl(\frac{\gamma_n(r)^n}{r},\frac{q}{r}\bigr)>1$ then we have $W=0$ and otherwise we multiply the above equation by the multiplicative inverse of $\frac{\gamma_n(r)^n}{r} \pmod{\frac{q}{r}}$ and retrieve a polynomial of the shape \eqref{form}. We remark that $z$ is in an interval of length at most $\frac{q^{1/n}}{\gamma_n(r)}\le \left(\frac{q}{r}\right)^{1/n}$ so that we have transformed the original problem into a problem that has the desired property.

Case 2: $\frac{q}{r} = 1$. In this case, since   $z$ is in an interval of length at most $\frac{q^{1/n}}{\gamma_n(r)}\le \left(\frac{q}{r}\right)^{1/n}=1$, we have at most two solutions and \eqref{thm_4} holds.

We are now ready to prove \eqref{thm_4} under the hypothesis $\gcd(d,q)=1$. We begin with the case $\omega(q) \ge 2$. Let $\W$ be the set of solutions to the equation \eqref{thmcong}. For each prime $p$ with $p^\alpha \| q$ we denote by $v_p$ the number of solutions to the equation $P(x) \equiv 0 \pmod{p^\alpha}$. Suppose at first that there is a prime number $p$ for which $p^\alpha \| q$ and $p^{\alpha} > q^{\frac{1}{\omega(q)+1}}$ ($p^{\alpha} = q^{\frac{1}{\omega(q)+1}}$ is impossible). From Lemma \ref{lem3}, the numbers $x \in \W$ are in at most $v_p \le n$ congruence classes modulo $p^{\alpha}$. Let's denote by $\W^\prime$ the set of solutions $x \in \W$ that are in one of the most popular congruence classes modulo $p^{\alpha}$. We write $s:=\#\W^\prime$, so that $W \le ns$. Now, set $q_1:=p^\alpha$ and $q_2:=\frac{q}{p^\alpha}$ and consider the product
$$
\Delta:=\prod_{\substack{x_1 < x_2 \\ x_1,x_2 \in \W^\prime}}(x_2-x_1)
$$
From Lemma \ref{lem4}, if $s \ge 2$, then $q_2^{\frac{s^2}{2n}-\frac{s}{2}} \mid \Delta$. Also, since $q_1^{\binom{s}{2}} \mid \Delta$, we must have
$$
q_1^{\binom{s}{2}}q_2^{\frac{s^2}{2n}-\frac{s}{2}} \le q^{\frac{1}{n}\binom{s}{2}}=(q_1q_2)^{\frac{1}{n}\binom{s}{2}},
$$
thus $q_1 \le q_2^{\frac{1}{s-1}}$. However, $q_1 > q^{\frac{1}{\omega(q)+1}} \Rightarrow q_1 > q_2^{\frac{1}{\omega(q)}}$. We deduce that
$$
q_2^{\frac{1}{\omega(q)}} < q_1 \le q_2^{\frac{1}{s-1}}
$$
so that $s \le \omega(q)$ and $W \le n\omega(q)$ if $s \ge 2$ and $W \le n$ otherwise.

Now, if such a prime number does not exist, it is because $2^\alpha \| q$ with $2^\alpha > q^{\frac{2}{\omega(q)+1}}$. The rest of the argument is similar except that $v_2 \le 2n$, so that $W \le 2ns$. If $s \ge 2$, we still come to the conclusion $q_1 \le q_2^{\frac{1}{s-1}}$ except that now $q_1 > q^{\frac{2}{\omega(q)+1}} \Rightarrow q_1 > q_2^{\frac{2}{\omega(q)-1}}$. We deduce that $s \le \frac{\omega(q)}{2}$, so that $W \le n\omega(q)$ if $s \ge 2$ and $W \le 2n$ otherwise. We have established \eqref{thm_4} in the case $\omega(q) \ge 2$.

We now assume that $\omega(q)=1$. Since $q^{1/n} \le q$, we deduce from Lemma \ref{lem3} that $q=2^\alpha$ for some $\alpha \ge 3$. Then, again from Lemma \ref{lem3}, we deduce that $n=2^k$ for some $k \ge 1$.

We first consider the case $n=2$. One can show with the help of the representation $x \equiv (-1)^a5^b \pmod{2^\alpha}$ (see the proof of Lemma \ref{lem3}) that if the equation \eqref{thmcong} has a solution, then it has 4 solutions and they are of the form $x \equiv \pm (2^{\alpha-1}+1)z \pmod{2^\alpha}$ for some $z \in \{1,\dots,2^{\alpha-2}-1\} \pmod{2^{\alpha}}$. The result \eqref{thm_4} follows from $\lfloor 2^{\alpha/2} \rfloor \le 2^{\alpha-2}$ for $\alpha \ge 3$.

We now turn to the case $n=2^k$ for some $k \ge 2$. Let's write
\begin{eqnarray*}
\T_1 & := & \{1,\dots,2^{\alpha-1}-1\} \pmod{2^\alpha}\\
\T_2 & := & \{2^{\alpha-1}+1,\dots,2^{\alpha}-1\} \pmod{2^\alpha}.
\end{eqnarray*}
Again, since every solution $x \in \T_1$ to \eqref{thmcong} has its associated solution $-x \in \T_2$, we deduce that \eqref{thm_4} holds if all the odd numbers in $\I$ are included in one of $\T_1$ or $\T_2$. If it is not the case, then since
$$
(x+2^{\alpha-2})^{2^k} \equiv x^{2^k} \pmod{2^\alpha}\qquad (k \ge 2,\ \alpha \ge 3),
$$
we deduce that the number of solutions to \eqref{thmcong} is the same with $\I$ replaced by $\I':=\I+2^{\alpha-2}$. Now, since $2^{\alpha/4} < 2^{\alpha-2}$ for $\alpha \ge 3$, we must have that all the odd numbers in $\I'$ are included in one of $\T_1$ or $\T_2$. The proof is completed.

\section{Concluding remarks}

It is interesting to consider Theorem \ref{thm:2} with $m = 2$. Let $\alpha:=a+bi$ with $a,b \in \mathbb{Z}$. The ideal $(\alpha) \subseteq \mathbb{Z}[i]$ can also be seen as the lattice $\Gamma$ generated by $v_1:=(a,b)$ and $v_2:=(-b,a)$ in $\mathbb{Z}^2$. The fundamental domain of $\Gamma$ is a square of area $N(\alpha)$ with the base $v_1,v_2$, precisely $\{\lambda_1 v_1+\lambda_2 v_2:\ \lambda_1,\lambda_2 \in [0,1)\}$. Here and throughout, the norm is $N(z_1+z_2 i) = z^2_1+z^2_2=\| z_1+z_2 i\|^2$ ($z_1,z_2 \in \mathbb{R}$) as usual.

Assume that we have a bounded set $\Omega \in \mathbb{R}[i]$ of diameter, defined by
$$
d(\Omega):=\sup_{\alpha_1,\alpha_2 \in \Omega}(N(\alpha_2-\alpha_1))^{1/2},
$$
nonzero. Assume also that we have a set $\S$ that contains $x_1,\dots,x_S$ elements of $\mathbb{Z}[i]$ and that we have a set $\mathcal{Q}$ of pairwise coprime elements of $\mathbb{Z}[i]$ such that for each $q \in \mathcal{Q}$ the elements of $\S$ are in at most $\nu(q)$ of the $N(\alpha)$ equivalence classes of $\mathbb{Z}[i]/(\alpha)$. Then, considering
$$
\Delta:=\prod_{1 \le i < j \le S}N(x_j-x_i)
$$
and arguing as in the proof of Theorem \ref{thm:1} we get to
$$
S \le \frac{\sum_{q \in \mathcal{Q}} \log N(q)-2\log d(\Omega)}{\sum_{q \in \mathcal{Q}}\frac{\log N(q)}{\nu(q)} - 2\log d(\Omega)}
$$
provided that the denominator is positive.

It is a refinement of Theorem \ref{thm:2} in a very special case. We are not aware of this kind of generalization in $\mathbb{R}^m$ for any $m \ge 3$.


\begin{thebibliography}{99}

\normalsize
\baselineskip=17pt

\bibitem{gea:ra:rr} G. E. Andrews, R. Askey and R. Roy, \emph{Special functions}, Encyclopedia of Mathematics and its Applications, 71, Cambridge University Press, Cambridge, 1999.

\bibitem{jb:lc} J. Brenner and L. Cummings, \emph{The Hadamard maximum determinant problem}, Amer. Math. Monthly 79 (1972), 626--630. 

\bibitem{esc:ce} E. S. Croot and C. Elsholtz, \emph{On variants of the larger sieve}, Acta Math. Hungar. 103 (2004), no. 3, 243--254. 

\bibitem{jf:hi} J. Friedlander and H. Iwaniec, \emph{Opera de cribro}, American Mathematical Society Colloquium Publications, 57 American Mathematical Society, Providence, RI, 2010.

\bibitem{pxg} P. X. Gallagher, \emph{A larger sieve}, Acta Arith. 18 (1971), 77--81.

\bibitem{ag:jju} A. Granville and J. Jim\'enez-Urroz, \emph{The least common multiple and lattice points on hyperbolas}, Quarterly Journal of Mathematics (Oxford), 51 (2000), 343--352.

\bibitem{svk} S. V. Konyagin, \emph{The number of solutions of congruences of the $n$th degree with one unknown} (Russian), Mat. Sb. (N.S.) 109(151) (1979), no. 2, 171--187.

\bibitem{svk:ts} S. V. Konyagin and T. Steger, \emph{Polynomial congruences} (Russian), Mat. Zametki 55 (1994), no. 6, 73--79, 158; translation in Math. Notes 55 (1994), no. 5-6, 596--600.

\bibitem{pl} P. Letendre, \emph{The number of integer points close to a polynomial}, Annales Univ. Sci. Budapest., Sec. Comp., 48 (2018), 81--93.

\end{thebibliography}
\end{document}